\documentclass[11pt]{amsart}
\usepackage{amsfonts,latexsym,amsthm,amssymb,graphicx,amscd}
\usepackage[all]{xy}
\usepackage[usenames]{color} 
\usepackage{amscd,epsfig}
\usepackage{hyperref}
\usepackage{pdfsync}
\usepackage{graphicx}
\usepackage{tableau}

\DeclareFontFamily{OT1}{rsfs}{}
\DeclareFontShape{OT1}{rsfs}{n}{it}{<-> rsfs10}{}
\DeclareMathAlphabet{\mathscr}{OT1}{rsfs}{n}{it}

\CompileMatrices

\newtheorem{theorem}{Theorem}[section]

\newtheorem{corol}[theorem]{Corollary}
\newtheorem{prop}[theorem]{Proposition}

{\theoremstyle{definition} }
{\theoremstyle{remark} \newtheorem{remark}[theorem]{Remark}
\newtheorem{example}[theorem]{Example}}

\numberwithin{equation}{section}

\newcommand{\A}{{\mathbb A}}

\newcommand{\C}{{\mathbb C}}
\newcommand{\Pbb}{{\mathbb{P}}}
\newcommand{\Q}{{\mathbb Q}}
\newcommand{\T}{{\mathbb T}}
\newcommand{\Z}{{\mathbb Z}}
\newcommand{\cF}{{\mathcal F}}
\newcommand{\conF}{{\sf F}}
\newcommand{\cA}{{\mathcal A}}
\newcommand{\cO}{{\mathcal O}}
\newcommand{\caD}{{\mathcal D}}
\newcommand{\cH}{{\mathcal H}}

\newcommand{\cM}{{\mathcal M}}
\newcommand{\caIC}{{\mathcal I\mathcal C}}
\newcommand{\one}{1\hskip-3.5pt1}
\newcommand{\chc}{{\check{c}}}
\newcommand{\chs}{{\check{s}}}
\newcommand{\csm}{{c_{\text{SM}}}}
\newcommand{\ssm}{{s_{\text{SM}}}}
\newcommand{\cful}{{c_{\text{F}}}}
\newcommand{\chcsm}{{\chc_{\text{SM}}}}
\newcommand{\chssm}{{\chs_{\text{SM}}}}
\newcommand{\cma}{{c_{\text{Ma}}}}
\newcommand{\sma}{{s_{\text{Ma}}}}
\newcommand{\cIH}{{c_{\text{IH}}}}
\newcommand{\SM}{s_*}
\newcommand{\chSM}{\chs_*}
\newcommand{\chcIH}{{\chc_{\text{IH}}}}
\newcommand{\chcma}{{\chc_{\text{Ma}}}}
\newcommand{\chsma}{{\chs_{\text{Ma}}}}
\newcommand{\chima}{{\chi_{\text{Ma}}}}
\newcommand{\chiIH}{{\chi_{\text{IH}}}}
\newcommand{\qede}{\hfill$\lrcorner$}
\DeclareMathOperator{\Mil}{{Mi}}
\DeclareMathOperator{\sMil}{{SMi}}
\DeclareMathOperator{\chMil}{{\check\Mil}}
\DeclareMathOperator{\chsMil}{{\check\sMil}}
\DeclareMathOperator{\Gr}{Gr}
\DeclareMathOperator{\Char}{Char}
\DeclareMathOperator{\CC}{CC}
\DeclareMathOperator{\Eu}{Eu}
\DeclareMathOperator{\chEu}{\check Eu}
\DeclareMathOperator{\Segre}{Segre}
\DeclareMathOperator{\mult}{mult}
\DeclareMathOperator{\DR}{DR}
\DeclareMathOperator{\Perv}{Perv}
\DeclareMathOperator{\IC}{IC}
\DeclareMathOperator{\im}{im}
\DeclareMathOperator{\Mod}{Mod}
\DeclareMathOperator{\supp}{supp}

\begin{document}
\title{Positivity of Segre-MacPherson classes}

\date{\today}

\author{Paolo Aluffi}
\address{
Mathematics Department, 
Florida State University,
Tallahassee FL 32306, USA
}
\email{aluffi@math.fsu.edu}

\author{Leonardo C.~Mihalcea}
\address{
Department of Mathematics, 
Virginia Tech University, 
Blacksburg VA 24061, USA
}
\email{lmihalce@vt.edu}

\author{J\"org Sch\"urmann}
\address{Mathematisches Institut, Universit\"at M\"unster, Germany}
\email{jschuerm@uni-muenster.de}

\author{Changjian Su} \address{Department of Mathematics, University
  of Toronto, Toronto ON M5S 1A1, Canada }
\email{changjiansu@gmail.com}


\thanks{L.~C.~Mihalcea was supported in part by NSA Young Investigator
  Award H98320-16-1-0013 and the Simons Collaboration Grant 581675;
  J. Sch\"{u}rmann was funded by the Deutsche Forschungsgemeinschaft
  (DFG, German Research Foundation) under Germany's Excellence
  Strategy -- EXC 2044--390685587, Mathematics M\"{u}nster: Dynamics
  -- Geometry -- Structure}

\begin{abstract}
Let $X$ be a complex nonsingular variety with globally generated
tangent bundle. We prove that the signed Segre-MacPherson (SM) class
of a constructible function on~$X$ with effective characteristic cycle
is effective. {This observation has a surprising number of applications to positivity 
questions in classical situations, unifying previous results in the
literature and yielding several new results. We survey a selection
of such results in this paper.} 
For example, we prove
general effectivity results for SM classes of subvarieties which admit 
proper (semi-)small resolutions and for regular or affine embeddings. 
Among these, we mention the effectivity of (signed) Segre-Milnor 
classes of complete intersections if $X$ is projective and an alternation 
property for SM classes of Schubert cells in flag manifolds; the latter 
result proves and generalizes a variant of a conjecture of Feh{\'e}r and 
Rim{\'a}nyi. Among other applications we prove the positivity of 
Behrend's Donaldson-Thomas invariant for a closed subvariety of an 
abelian variety and the signed-effectivity of the intersection homology 
Chern class of the theta divisor of a non-hyperelliptic curve;
and we extend the (known) 
non-negativity of the Euler characteristic of perverse sheaves on a 
semi-abelian variety to more general varieties dominating an abelian 
variety.\end{abstract}

\maketitle

\section{Introduction}\label{sec:intro}

\subsection{}
In this note, $X$ will denote a nonsingular complex variety and
$Z\subseteq X$ will be a closed subvariety; here (sub)varieties are by
definition irreducible and reduced.  We will assume that the tangent
bundle of $X$ is globally generated. In the projective case, this is
equivalent to asking that $X$ be a projective homogeneous
variety---for example a projective space, a flag manifold, or an
abelian variety; but our main results will hold in the non-complete
case as well.  We denote by $A_*(Z)$ the Chow group of 
{cycles on $Z$ modulo rational equivalence}, and by
$\conF(Z)$ the group of constructible functions on $Z$; here we allow
$Z$ to be more generally a closed reduced subscheme of $X$.
{Our general aim is to investigate the {\em positivity\/} of certain rational 
equivalence classes associated with the embedding of $Z$ in $X$, or 
more generally with suitable constructible functions on $Z$. 
We generalize several known positivity results on e.g.,
Euler characteristics, and provide a framework leading to analogous
results in a broad range of situations.}
Answering a conjecture of Deligne and Grothendieck,
MacPherson~\cite{macpherson:chern} constructed a group homomorphism
$c_*: \conF(Z) \to A_*(Z)$ which commutes with proper push-forwards
and satisfies a normalization property: if $Z$ is non-singular, then
$c_*(\one_Z) = c(TZ) \cap [Z] $, where $c(TZ)$ is the total Chern
class of $Z$. (MacPherson worked in homology;
see~\cite[Example~19.1.7]{fulton:IT} for the refinement of the theory
to the Chow group.)  If $Y \subseteq Z$ is a constructible subset, the
Chern-Schwartz-MacPherson (CSM) class $\csm(Y)\in A_*(Z)$ is the image
$c_*(\one_Y)$ of the indicator function of $Y$ under MacPherson's
natural transformation. Let $\varphi \in \conF(Z)$. We will focus on
the closely related Segre-MacPherson (SM) class
\[
\SM(\varphi,X):= c(TX|_Z)^{-1}\cap c_*(\varphi) \in A_*(Z) \/. 
\] 
(The class $c(TX|_{Z})$ is invertible in $A_*(Z)$, because it is of
the form $1+ a$, where $a$ is nilpotent.)  In particular, we let
$\ssm(Y,X)$ denote the Segre-Schwartz-MacPherson (SSM) class
$\SM(\one_Y,X)\in A_*(Z)$; note that this class depends on both $Y$
and the ambient variety~$X$. If $Y$ is a subvariety of $Z$, then the
top-degree component of $\ssm(Y,X)$ in $A_{\dim Y}(Z)$ is the
fundamental class $[\overline{Y}]$ of the closure of $Y$.  Further, if
$Y=Z$ is a {\em nonsingular\/} closed subvariety of $X$, then $\ssm(Y,X)\in
A_*Y$ equals the ordinary Segre class $s(Y,X)$; in general, the two
classes differ (as discussed in subsection~\ref{Milnor classes} for
$Y$ a global complete intersection in a nonsingular projective variety
$X$).  See~\cite{aluffi:IE1} and (in the equivariant
case)~\cite{ohmoto:eqcsm} for general properties of SM classes,
and~\cite{schurmann:transversality} for their compability with
transversal pullbacks.

There are `signed' versions of both $c_*$ and $\SM$ (respectively, 
$\csm$ and $\ssm$), which appear naturally when relating them to 
characteristic cycles. If $c_*(\varphi) = c_0+ c_1 + \ldots $ is the decomposition into
homogeneous components (i.e., $c_i \in A_i(Z)$) then the `signed'
class $\chc_*(\varphi)$ is defined by
\[ 
\chc_*(\varphi) = c_0 - c_1 + c_2 - \ldots \quad \text{and} \quad
\chcsm(Y):=\chc_*(\one_Y) \/;
\] 
that is, by changing the sign of each homogeneous component of odd 
dimension. One defines similarly the signed SM class 
\[
\chSM(\varphi,X):=  c(T^*X|_Z)^{-1} \cap
\chc_*(\varphi) \quad \text{and} \quad \chssm(Y,X):=\chSM(\one_Y,X)\/.
\] 

A basis for $\conF(Z)$ consists of the {\em local Euler
obstructions\/} $\Eu_Y$ for closed subvarieties $Y$ of~$Z$. In fact,
the characteristic cycle of the (signed) local Euler obstruction is an
irreducible Lagrangian cycle in $T^*X$, and from this perspective the
functions $\Eu_Y$ are the `atoms' of the theory; see
equation~\eqref{eq:CC} below. If $Y$ is nonsingular,
$\Eu_Y=\one_Y$. The local Euler obstruction is a subtle and
well-studied invariant of singularities (see e.g.,
\cite{macpherson:chern, MR647684, MR634426, MR1301184, MR1783853,
  MR2600874}).  The corresponding class $c_*(\Eu_Z)\in A_*(Z)$ for $Z$
a subvariety of $X$ is the {\em Chern-Mather class\/} of $Z$,
$\cma(Z)=\nu_*(c(\tilde{T})\cap[\tilde{Z}])$, with $\nu: \tilde{Z}\to
Z$ the {\em Nash blow-up\/} of $Z$ and $\tilde{T}$ the tautological
bundle on $\tilde{Z}$ extending $TZ_{reg}$,
cf.~\cite{macpherson:chern} or~\cite[Example~4.2.9]{fulton:IT}. In
particular $\cma(Z)=c(TZ)\cap [Z]$ if $Z$ is nonsingular. If $Z$ is
complete, we denote by $\chima(Z):=\chi(Z,\Eu_Z)$ the degree
of~$\cma(Z)$; so $\chima(Z)$ equals the usual topological Euler
characteristic $\chi(Z)$ if $Z$ is nonsingular and complete. We also
consider the corresponding Segre-Mather class $\sma(Z,X):=
\SM(\Eu_Z,X)$ as well as the signed classes
\[ 
\chcma(Z):= \chc_*(\Eu_Z); \quad \chsma(Z,X) := c(T^*X|_Z)^{-1}
\cap \chc_*(\Eu_Z) \/.  
\]
With our conventions, we get $(-1)^{\dim Z}\chcma(Z)
=\nu_*(c(\tilde{T}^*)\cap[\tilde{Z}])$ in terms of the dual
tautological bundle on the Nash blow-up (which differs by the sign
$(-1)^{\dim Z}$ from the definition of the signed Chern-Mather class
used in some references like~\cite{sabbah:quelques,
  schurmann.tibar}).

The main result in this paper is the 
{(signed)} 
effectivity of Segre-MacPherson
classes in a large class of examples.  By an {\em effective\/} class
we mean a class which can be represented by a nonzero, non-negative,
cycle.

\begin{theorem}\label{thm:mather}
Let $X$ be a complex nonsingular variety, and assume that the tangent
bundle~$TX$ is globally generated. Let $Z\subseteq X$ be a closed
subvariety of $X$. Then the following hold:

(a) The class $(-1)^{\dim Z}\chsma(Z,X)\in
A_*(Z)$ is effective.

(b) Assume that the inclusion $U \hookrightarrow Z$ is an affine
morphism, where $U$ is a locally closed smooth subvariety of $Z$. Then
$(-1)^{\dim U}\chssm(U,X)\in A_*(Z)$ is effective.\end{theorem}

If in addition $X$ is assumed to be complete, then the requirement that
$TX$ is globally generated is equivalent to $X$ being a homogeneous
variety; cf.~e.g., \cite[Corollary~2.2]{brion:spherical}. Further,
Borel and Remmert \cite{borel.remmert:uber} (see also~\cite[Theorem
  2.6]{brion:spherical}) prove that all complete homogeneous varieties
are products $(G/P) \times A$, where $G$ is a semisimple Lie group, $P
\subseteq G$ is a parabolic subgroup, and $A$ is an abelian variety.

{Theorem~\ref{thm:mather} is extended to more general constructible
functions in Theorem~\ref{thm:ssmpos}, below.} These theorems
may be used to prove several positivity statements, unifying and
generalizing analogous results from the existing literature. We
list below the situations
{which we will highlight in this paper to illustrate 
applications of our methods,}
and the sections where these are discussed.
\begin{enumerate}
\item[(a)] Closed subvarieties of abelian varieties; primarily in
  \S\ref{s:abelian}.
\item[(b)] The proof of a generalization of a
conjecture of Feh{\'e}r and Rim{\'a}nyi~\cite{feher.rimanyi:csmdeg}
concerning SSM classes of Schubert cells in Grassmannians; \S\ref{ss:FR}.
\item[(c)] Complements of hyperplane arrangements; \S\ref{ss:hyp}
\item[(d)] Positivity of certain Donaldson-Thomas type invariants;
  \S\ref{ss:DT}.
\item[(e)] Intersection homology Segre and Chern classes; \S\ref{s:charIH}.
\item[(f)] Semi-small resolutions; \S\ref{ss:ssmall}.
\item[(g)] Regular embeddings and Milnor classes; \S\ref{Milnor classes}.
\item[(h)] Semi-abelian varieties and generalizations; \S\ref{Semi-abelian}.
\end{enumerate}

{Ultimately, these positivity statements follow from the effectivity
of the associated characteristic cycles. In~\S\ref{s:poschar} we survey a more
comprehensive list of situations in which the characteristic cycle is positive.}

The proof of Theorem~\ref{thm:mather} will be given in
\S\ref{ss:proof} below. It is suprisingly easy, but not elementary;
it is based on a classical formula by Sabbah
\cite{sabbah:quelques}, calculating the (signed) CSM class of a
constructible function $\varphi$ in terms of its characteristic cycle
$\CC(\varphi)$; see Theorem~\ref{thm:checkssm} below.

{\em Acknowledgements.\/} P.A.~thanks the University of Toronto for
the hospitality.  L.M.~is grateful to R.~Rim{\'a}nyi for many
stimulating discussions about the CSM and SSM classes, including the
positivity conjecture from \cite{feher.rimanyi:csmdeg}, and to the
Math Department at UNC Chapel Hill for the hospitality during a
sabbatical leave in the academic year 2017-18. The authors are
grateful to an anonymous referee whose comments prompted a
reorganization of this paper.

Finally, this paper is dedicated to Professor William Fulton with the
occasion of his 80th birthday. His interest in positivity questions arising in algebraic geometry, and his influential ideas, continue to inspire us.
\section{Characteristic classes via characteristic cycles; proof of 
the main theorem}\label{sec:ccvcc}

\subsection{Characteristic cycles}
Let $X$ be a smooth complex variety. We recall a commutative diagram
which plays a central role in 
{seminal work of Ginzburg}
\cite{ginzburg:characteristic}; it is largely based on results
from~\cite{beilinson.bernstein:localisation,
  brylinski.kashiwara:KL,kashiwara.tanisaki:characteristic}. We also
considered this diagram in our previous work \cite[\S6]{AMSS:shadows},
and we use the notation from this reference.

\begin{equation}\label{E:diagram}
\vcenter{\xymatrix@C=40pt{
\Perv(X) \ar[d]_{\chi_{stalk}} & \Mod_{rh}(\caD_X) \ar[l]_{\DR}^\sim 
\ar[d]^\Char \\
\conF (X) \ar[r]^\CC_\sim & L(X) 
}}
\end{equation}
Here $\Mod_{rh}(\caD_X)$ denotes the Abelian category of algebraic
holonomic $\caD_X$-modules with regular singularities, and $\Perv(X)$
is the Abelian category of perverse (algebraically) constructible
complexes of sheaves of $\C$-vector spaces on $X$; $\conF(X)$ is the
group of constructible functions on $X$ and $L(X)$ is the group of
conic Lagrangian cycles in $T^*X$.  The functor $\DR$ is defined on $M
\in \Mod_{rh}(\caD_X)$ by
\[
\DR(M) = R\cH om_{\caD_X}(\cO_X,M)[\dim X]\/,
\]
that is, it computes the DeRham complex of a holonomic module (up to a
shift), viewed as an {\em analytic\/} $\caD_X$-module. This functor
realizes the Riemann-Hilbert correspondence, and is an equivalence. We
refer to e.g., \cite{kashiwara.tanisaki:characteristic,
  ginzburg:characteristic} for details. The left map $\chi_{stalk}$
computes the stalkwise Euler characteristic of a constructible
complex, and the right map $\Char$ gives the characteristic cycle of a
holonomic $\caD_X$-module. The map $\CC$ is the characteristic cycle
map for constructible functions; if $Z \subseteq X$ is closed and
irreducible, then
\begin{equation}\label{eq:CC}
\CC (\Eu_Z) = (-1)^{\dim Z} [T^*_Z X] \/;
\end{equation} 
here $\Eu_Z$ is the local Euler obstruction (see~\S\ref{sec:intro}),
and $T^*_Z X := \overline{T^*_{Z^{reg}} X}$ is the conormal space of
$Z$, i.e., the closure of the conormal bundle of the smooth locus of
$Z$. The commutativity of diagram~\eqref{E:diagram} is shown in
\cite{ginzburg:characteristic} using deep $\caD$-module techniques; it
also follows from \cite[Example~5.3.4, p.~359--360]{schurmann:book}
(even for a holonomic $\caD$-module without the regularity
requirement).  Also note that the upper transformations in
\eqref{E:diagram} factor over the corresponding Grothendieck groups,
so they also apply to complexes of such $\caD$-modules. If $f:X \to Y$
is a proper map of smooth complex varieties, there are well-defined
push-forwards for each of the objects in the diagram, denoted by
$f_*$. Furthermore, all the maps commute with proper push forwards;
cf.~\cite[Appendix]{ginzburg:characteristic}. For others proofs, 
see~\cite[Proposition~4.7.5]{HTT} for the transformation $\DR$, 
\cite[\S2.3]{schurmann:book} for the transformation $\chi_{stalk}$ and
\cite[\S4.6]{schurmann:lectures} for the transformation $\CC$ (for
the transformation $\Char$ it then follows from the commutativity of
diagram~\eqref{E:diagram}).

The next result, relating characteristic cycles to (signed) CSM
classes, has a long history. See \cite[Lemme 1.2.1]{sabbah:quelques},
and more recently \cite[(12)]{PP:hypersurface},
\cite[\S4.5]{schurmann:lectures},
\cite[\S3]{schurmann:transversality}, especially diagram (3.1)
in~\cite{schurmann:transversality}.

\begin{theorem}\label{thm:checkssm} 
Let $X$ be a complex nonsingular variety, and let $Z\subseteq X$ be a
closed reduced subscheme.  Let $\varphi \in \conF(Z)$ be a
constructible function on $Z$. Then
\[ 
\chc_*(\varphi) = c(T^*X|_Z) \cap \Segre(\CC(\varphi)) 
\] 
as elements in the Chow group $A_*(Z)$ of $Z$. Here
$\Segre(\CC(\varphi))$ is the Segre class associated to the conic
Lagrangian cycle $\CC(\varphi) \subseteq T^*X|_Z$.
\end{theorem}

We recall the definition of the Segre class used in
Theorem~\ref{thm:checkssm}. Let $q: \Pbb(T^*X|_Z \oplus \one) \to Z$
be the projection from the restriction of the projective completion of
the cotangent bundle of $X$. If $C \subseteq T^*X|_Z$ is a cone
supported over $Z$, and $\overline{C}$ is the closure in $\Pbb(T^*X|_Z
\oplus \one)$, the Segre class is defined by
\[ 
\Segre(C): = q_* \bigl(
  \sum_{i \ge 0} c_1(\cO_{\Pbb(T^*X|_Z \oplus \one)}(1))^i \cap 
[\overline{C}] 
  \bigr) 
\]
as an element of $A_*(Z)$; see \cite[\S4.1]{fulton:IT}.

Every irreducible conic Lagrangian subvariety of $T^*X$ is a
conormal cycle $T^*_Z X$ for $Z \subseteq X$ a closed subvariety; see
e.g.,~\cite[Theorem~E.3.6]{HTT}. From this it follows that
every non-trivial characteristic cycle is a
linear combination of conormal spaces:
\begin{equation}\label{E:CCexp}
\CC(\varphi)=\sum_Y a_Y [T^*_Y X]
\end{equation}
for uniquely determined closed subvarieties $Y$ of $Z$ and nonzero
integer coefficients $a_Y$. By~\eqref{eq:CC}, the coefficients $a_Y$
are determined by the equality of constructible functions
\begin{equation}\label{eq:eqconst}
0\neq \varphi = \sum_Y a_Y (-1)^{\dim Y}\Eu_Y \/.
\end{equation}

\subsection{Proof of the main theorem}\label{ss:proof} 
The following result is at the root of all applications in this note.

\begin{theorem}\label{thm:ssmpos}
Let $X$ be a complex nonsingular variety such that $TX$ is globally
generated. Let $Z\subseteq X$ be a closed reduced subscheme of $X$ and
let $\varphi \in \conF(Z)$ be a constructible function on~$Z$ such
that the characteristic cycle $\CC(\varphi) \in A_*(T^*X|_Z)$ is
effective.

Then $\chSM(\varphi,X)$ is effective in~$A_*(Z)$.  If $TX$ is trivial
(e.g.,~$X$ is an abelian variety), then $\chc_*(\varphi)$ is
effective.
\end{theorem}

\begin{proof}{By Theorem~\ref{thm:checkssm},
\[
\chSM(\varphi,X) = c(T^*X|_Z)^{-1}\cap 
\chc_*(\varphi) =\Segre(\CC(\varphi))\/.
\]
By hypothesis we have a decomposition~\eqref{E:CCexp} with positive
coefficients $a_Y$. It follows from~\eqref{eq:eqconst} that the Segre class
of $\CC(\varphi)$ is a linear combination of Segre classes of subvarieties:
\[ 
\Segre(\CC(\varphi)) = \sum_Y a_Y (-1)^{\dim Y} \Segre(\CC(\Eu_Y)) 
= \sum_Y a_Y (-1)^{\dim Y} \chsma(Y,X) \/.
\] 
By definition, the top degree part of each signed Segre-Mather class
$(-1)^{\dim Y} \chsma(Y,X)$ equals $[Y]$. Then the top degree part of
$\Segre(\CC(\varphi))$ equals a positive linear combination of those 
fundamental classes $[Y]$ of maximal dimension, and in particular
$\chSM(\varphi,X) = \Segre(\CC(\varphi))$ is not zero.}
Since the tangent bundle $TX$ is globally generated, it follows that
the line bundle $\cO_{\Pbb(T^*X \oplus \one)}(1)$ is globally
generated, as it is a quotient of $TX \oplus \one$.  Therefore its
first Chern class preserves non-negative classes.  Since
non-negativity is preserved by proper push-forwards, we can conclude
that under the given hypotheses $\Segre(\CC(\varphi))$ is
non-negative, and this completes the proof.
\end{proof}

Theorem~\ref{thm:mather} follows from Theorem~\ref{thm:ssmpos}:

\begin{proof}[Proof of Theorem~\ref{thm:mather}]
{
By Theorem~\ref{thm:ssmpos} it suffices to show that the
characteristic cycles for the constructible functions $(-1)^{\dim Z}
\Eu_Z$ and $(-1)^{\dim U} \one_U$ are effective. 
If $\varphi=\chEu_Z :=(-1)^{\dim Z}\Eu_Z$, then $\CC(\varphi)$ 
equals the conormal cycle $[T^*_Z X]$ of~$Z$, and it is therefore 
trivially effective.

Consider then $(-1)^{\dim U}\one_U$, and let $j:U \hookrightarrow X$
be the inclusion. We use the Riemann-Hilbert correspondence 
(diagram~\eqref{E:diagram}) to express the characteristic cycle.
By definition,
\[
(-1)^{\dim U} \one_U = \chi_{stalk}(j_!\C_U[\dim U])\/.
\]
Since $U$ is nonsingular, the sheaf $\C_U[\dim U]$ is perverse, and
$\cO_U$ is the corresponding regular holonomic $\caD_U$-module. 
We have $j_!\C_U[\dim U]=\DR(j_!(\cO_U))$; since $j$ is an affine 
morphism, $j_!(\cO_U)$ is a single regular holonomic 
$\caD_X$-module (with support in $Z$); see~\cite[p.~95]{HTT}.  
As pointed out in~\cite[p.~119]{HTT}, the characteristic cycles of 
non-trivial holonomic $\caD_X$-modules are effective, and this 
finishes the proof.}
\end{proof}

\begin{remark}\label{rmk:perverseU} 
{As the proof shows,}
the hypothesis that $U$ is smooth in Theorem~\ref{thm:mather}(ii) can
be weakened, by only requiring that that $\C_U[\dim U]$ is a perverse
sheaf. The proof of the effectivity then uses the fact that for an affine
inclusion $j: U \hookrightarrow Z$, {$j_!\C_U[\dim U]$} is a perverse sheaf
on {$Z$} 
(\cite[Lemma~6.0.2, p.~384 and Theorem~6.0.4, p.~409]{schurmann:book}).
{We will formalize this conclusion below, in 
Proposition~\ref{prop:holon}.}
For the case in which
$U=X\smallsetminus D$ is the open complement of a hypersurface $D$ in
$Z:=X$, the result also follows from~\cite[Proposition~6.0.2,
p.~404]{schurmann:book}.
\qede\end{remark}

\subsection{Effective characteristic cycles (I) } 
The applications in the rest of the paper follow from Theorem~\ref{thm:ssmpos}: they
represent situations when the characteristic cycle $\CC(\varphi)$ is
effective. 

{As pointed out above, every non-trivial characteristic cycle is a
linear combination of conormal spaces \eqref{E:CCexp}, and the
coefficients $a_Y$ in a linear combination are determined by the 
equality of constructible functions~\eqref{eq:eqconst}.}
The characteristic cycle of $\varphi\neq 0$ is effective if and only
if the coefficients $a_Y$ are positive. In particular, this condition
is intrinsic to the constructible function $\varphi\in \conF(Z)$ and
does not depend on the chosen closed embedding of $Z$ into an ambient
nonsingular variety $X$. 

A key source of examples where $\CC(\varphi)$ is effective, but
possibly reducible, arises as follows.  Constructible functions may be
associated with (regular) holonomic $\caD$-modules and perverse
sheaves $\cF\in \Perv(Z)$, cf.~diagram~\eqref{E:diagram}; for example, 
in the latter case the value
of the constructible function $\varphi:=\chi_{stalk}(\cF)$ at the
point $z\in Z$ is the Euler characteristic $\varphi(z)=\chi(\cF_z)$ of
the stalk at $z$ of the given complex of sheaves $\cF$.

\begin{prop}\label{prop:holon}
Let $X$ be a complex nonsingular variety such that $TX$ is globally
generated, and let $Z\subseteq X$ be a closed reduced subscheme.  Let
$0\neq \varphi\in \conF(Z)$ be a non-trivial constructible function
associated with
{a regular holonomic $\caD_X$-module supported on~$Z$, 
or (equivalently) a perverse sheaf on $Z$.}
Then $\chSM(\varphi,X)$ is effective in~$A_*(Z)$.
\end{prop}

\begin{proof}
This follows from the argument used in the proof of 
Theorem~\ref{thm:ssmpos}: the main observation is that the
characteristic cycle of a non-trivial (regular) holonomic $\caD$-module 
is effective; see e.g.,~\cite[p.~119]{HTT}. Further, perverse sheaves 
correspond to regular holonomic $\caD$-modules by means of the 
Riemann-Hilbert correspondence (see e.g.,~\cite[Theorem~7.2.5]{HTT}), 
compatibly with the construction of the associated constructible functions 
and characteristic cycles; cf.~diagram~\eqref{E:diagram}.
\end{proof}

There are situations where the characteristic cycle associated to a
constructible sheaf is known to be irreducible: examples include
characteristic cycles of the intersection cohomology sheaves of
Schubert varieties in the Grassmannian~\cite{MR1084458}, in more
general minuscule spaces~\cite{MR1451256}, of certain determinantal
varieties~\cite{zhang:chern}, and of the theta divisors in the
Jacobian of a non-hyperelliptic curve~\cite{MR1642745}.  In all such
cases, $\chSM(\varphi,X)$ is effective provided that $TX$ is globally
generated, by Theorem~\ref{thm:ssmpos}.  Also note that for the
varieties $Z$ listed above, the Chern-Mather class $\cma(Z)$ equals
$\cIH(Z)$, the intersection homology class defined in \S\ref{s:charIH}
below. This follows because in this case the characteristic cycle of
$\caIC_Z$ is irreducible, thus it must agree with the conormal cycle
of $Z$.

In the next few sections we discuss specific applications of
Theorem~\ref{thm:ssmpos} for various choices of the variety $X$ or
constructible function $\varphi$. The sections are mostly logically
independent of each other, and the reader may skip directly to the
case of interest. The only exception are the results concerning
Abelian varieties; these will be mentioned throughout this note.

A more detailed discussion on effective characteristic cycles is
given in \S\ref{s:poschar} below, including a more comprehensive
list of constructible
functions $\varphi$ for which
$\CC(\varphi)$ is effective,
and operations on characteristic functions which preserve the
effectivity of the corresponding characteristic cycles.

\section{Abelian varieties}\label{s:abelian}
If $X$ is an abelian variety, then $TX$ is trivial.  (In fact, this
characterizes abelian varieties among complete varieties,
cf.~\cite[Corollary~2.3]{brion:spherical}.) If $TX$ is trivial, then
for all constructible functions $\varphi$ on $Z$ the signed SM class
agrees with the signed CM class: $\chSM(\varphi,X)= \chc_*(\varphi)\in
A_*(Z)$. In particular, $\chsma(Z,X)= \chcma(Z)$ for $Z$ a subvariety
of $X$.  The following result follows then immediately from
Theorems~\ref{thm:mather} and~\ref{thm:ssmpos}.

\begin{corol}\label{corol:abelmather}
Let $Z$ be a closed subvariety of a smooth variety $X$ with $TX$
trivial (for example, an abelian variety). Then $(-1)^{\dim Z}
\chcma(Z)$ is effective.

More generally, let $\varphi$ be a constructible function on $Z$ such
that $\CC(\varphi)$ is effective. Then $\chc_*(\varphi) \in A_*(Z)$ is
effective.
\end{corol} 

As an example, the total Chern-Mather class
$\cma(\Theta)$ of the theta divisor in the Jacobian of a nonsingular
curve must be signed-effective.

Corollary~\ref{corol:abelmather} implies that
$\chi(Z,\varphi) \ge 0$, which also follows
from~\cite[Theorem~1.3]{MR1769729}. In particular, if $Z$ is a closed
subvariety of an abelian variety, then
\[
(-1)^{\dim Z}\chima(Z)=(-1)^{\dim Z}\chi(Z,\Eu_Z)\ge 0 \:.
\]
For nonsingular subvarieties $Z$, the Euler obstruction $\Eu_Z$ equals
$\one_Z$.~Then the fact that $(-1)^{\dim Z}\chi(Z)\ge 0$ is proven (in
the more general semi-abelian case) in~\cite[Corollary~1.5]{MR1769729}
(also see~\cite[(2)]{elduque.geske.maxim}). 
We note that the
  fact that $c(T^*Z)\cap [Z]$ is effective if $Z$ is a {\em
    nonsingular\/} subvariety of an abelian variety~$X$ also follows
  immediately from the fact that $T^*Z$ is globally generated, as it
  is a homomorphic image of the restriction of $T^*X$, which is
  trivial.  
Corollary~\ref{corol:abelmather} extends this result to 
{\em arbitrarily singular\/} closed subvarieties of a smooth
variety $X$ with trivial tangent bundle.  

In fact,
Corollary~\ref{corol:abelmather} also follows from Propositions~2.7
and~2.9 from \cite{schurmann.tibar}, where explicit effective cycles
representing~$\chSM(\varphi,X)$ in terms of suitable `polar classes'
are constructed. 


\section{Affine embeddings}\label{s:affine} 
An important family of positivity statements focus on the indicator
function $\one_U$ of a typically nonsingular and noncompact subvariety
$U$ of $Z\subseteq X$. In this case, among our main applications is
the proof a conjecture of Feh{\'e}r and Rim{\'a}nyi about the the
effectivity of SSM classes of Schubert cells.

Throughout this section we impose the hypotheses of
Theorem~\ref{thm:mather}(b), i.e., $X$ is a complex nonsingular
variety with $TX$ globally generated, $Z\subseteq X$ is a closed
subvariety of $X$, and the inclusion $U \hookrightarrow Z$ is an
affine morphism, with $U$ a locally closed smooth subvariety of~$Z$. 
Recall that
\[ 
\ssm(U,X) := c(TX|_Z)^{-1} \cap c_*(\one_U); \quad 
\chssm(U,X) := c(T^*X|_Z)^{-1} \cap \chc_*(\one_U) 
\] 
denote the SSM and the signed SSM classes associated to $U$. By
Theorem~\ref{thm:mather} (b),
\[ 
(-1)^{\dim U}\chssm(U,X)\in A_*(Z) 
\] 
is effective. If~$TX$ is trivial, then $(-1)^{\dim U}\chcsm(U)$ is
effective, so in particular for $X$ an abelian variety this implies
that $(-1)^{\dim U}\chi(U) \ge 0$.

\subsection{Schubert cells in flag manifolds and a conjecture of 
Feh{\'e}r and Rim{\'a}nyi}\label{ss:FR} 
Theorem~\ref{thm:mather} applies in particular if $U:=X(u)^\circ$ is a
Schubert cell in a flag manifold $X=G/P$, where $G$ is a complex
simple Lie group and $P$ is a parabolic subgroup. For example, $X$
could be a Grassmannian, or a complete flag manifold. Here $u\in W$ is
a minimal length representative for its coset in $W/W_P$, where $W$ is
the Weyl group of $G$ and $W_P$ is the Weyl group of $P$.  The
Schubert {\em cell\/}~$X(u)^\circ$ is defined to be $BuB/P$, where
$B\subseteq P$ is a Borel subgroup. It is well known that $X(u)^\circ
\cong \C^{\ell(u)}$, where $\ell(u)$ denotes the length of $u$.  The
closure $X(u)$ of $X(u)^\circ$ is the corresponding Schubert {\em
  variety,\/} and $X(u)= \bigsqcup_{w\leq u}\:X(w)^\circ$. We refer to
e.g., \cite{brion:flagv} for further details on these
definitions. Since the inclusion $X(u)^\circ \subseteq X$ is affine,
we obtain the following result.

\begin{corol}\label{cor:Schubert}
Let $X(u)^\circ$ be a Schubert cell in a generalized flag manifold
$G/P$. Then the class $(-1)^{\ell(u)} \chssm(X(u)^\circ,G/P)\in$
$A_*(X(u))$ is effective.
\end{corol}

Recall that $A_*(G/P)$ (resp., $A_*(X(u))$) has a
$\Z$-basis given by fundamental classes $[X(v)]$ of Schubert varieties
(with~$X(v)\subseteq X(u)$, i.e., $v\leq u$).  With this understood,
Corollary~\ref{cor:Schubert} may be rephrased as follows.

\begin{corol}\label{cor:SSMaltgp} 
Let $u \in W$ and consider the Schubert expansion
\[ 
\ssm(X(u)^\circ, G/P) = \sum a(w;u) [X(w)]
\]
with $a(w; u)\in \Z$. Then $(-1)^{\ell(u)- \ell(w)} a(w;u) \ge 0 $ for
all $w$.
\end{corol}

A similar positivity statement was conjectured by Feh{\'e}r and
Rim{\'a}nyi in~\S1.5 and Conjecture~8.4 of the paper
\cite{feher.rimanyi:csmdeg}. Their conjecture is stated for certain
degeneracy loci in quiver varieties, and in the `universal' situation
where the ambient space is a vector space with a group action. The
Schubert cells and varieties in the flag manifolds of Lie type~A are
closely related to a compactified version of such quiver
loci\footnote{One example are the matrix Schubert varieties, regarded
  in the space of all matrices. The study of those of maximal rank is
  closely related to Schubert varieties in the Grassmannian.}.  After
passing to the compactified version of the statements from
\cite{feher.rimanyi:csmdeg}, Corollary~\ref{cor:SSMaltgp} proves the
conjecture from \cite{feher.rimanyi:csmdeg} in the Schubert instances;
see \cite[\S6 and \S7]{feher2018motivic} for a comparison between the
`universal' and `compactified' versions. A specific comparison between
our calculations and those from~\cite{feher.rimanyi:csmdeg} is
included in the following example. We note that in arbitrary Lie type
a description of Schubert varieties via quiver loci is not available.

\begin{example}\label{ex:FR} 
Let $X= \Gr(2,5)$ be the Grassmann manifold parametrizing subspaces of
dimension $2$ in $\C^5$. In this case one can index the Schubert cells
by partitions included in the $2 \times 3$ rectangle, such that each
cell has dimension equal to the number of boxes in the partition. With
this notation, and using the calculation of CSM classes of Schubert
cells from \cite{aluffi.mihalcea:csm}, one obtains the following
matrix encoding Schubert expansions of SSM classes of Schubert cells:

\[ 
\begin{pmatrix} 
1 & - 4 & 5 & 4 & -2 & -10 & 5 & 4 & -4 & 1 \\ 
0 & 1 & -3 & -3 & 2 & 10 & -7 & -5 & 7 & -2 \\ 
0 & 0 & 1 & 0 & -2 & -3 & 7 & 3 & -9 & 3 \\ 
0 & 0 & 0 & 1 & 0 & -3 & 2 & 2 & -3& 1 \\ 
0 & 0 & 0 & 0 & 1 & 0 & -3 & 0 & 3 & -1\\ 
0 & 0 & 0 & 0 & 0 & 1 & -2 & -2 & 5 & -2 \\ 
0 & 0 & 0 & 0 & 0& 0 & 1 & 0 & -2 & 1\\ 
0 & 0 & 0 & 0 & 0 & 0 & 0 & 1 & -2 & 1\\ 
0 & 0 & 0 & 0 & 0 & 0 & 0 & 0 & 1 & -1\\ 
0 & 0 & 0 & 0 & 0 & 0 & 0 & 0 & 0 & 1 
\end{pmatrix} \]

The columns, read left to right, and rows, read top to bottom, are
indexed by:
\[ 
\emptyset\,, \tableau{5}{{}}\,, \tableau{5}{{}&{}}\,, \tableau{5}{{} \\ {}}\,,
\tableau{5}{{}& {} & {}}\,,  \tableau{5}{& {}\\ {}& {}}\,,  
\tableau{5}{& & {}\\ {}& {}&  {}}\,, \tableau{5}{{}&{}\\ {}&{}}\,, 
\tableau{5}{& {} & {} \\ {} & {} & {}}\,, \tableau{5}{{}& {}&{}\\ {} & {}&{}} ~\/. 
\]

After taking duals in the $2 \times 3$ rectangle, these give the same
coefficients as in equation~(3) from \cite{feher.rimanyi:csmdeg}. (The
calculations in \cite{feher.rimanyi:csmdeg} are done in a stable
limit, therefore for our purposes one disregards partitions not
included in the given rectangle.) Another example is given by the
calculation of the SSM class for the partition $(3,1)$ in $\Gr(2,6)$:
\[ 
\ssm((\tableau{5}{ & & {} \\ {} & {} & {} })^\circ)
= \tableau{5} { & & {} \\ {} & {} & {} }
- 3~\tableau{5} { & {} \\ {} & {} } - 4~\tableau{5} { {} & {} & {} }
+ 13~\tableau{5} { {} & {} } + 5~\tableau{5} { {} \\ {} }
- 22~\tableau{5}{ {} } 
+22~\emptyset \/.
\]
(Here $\lambda$ denotes the Schubert class indexed by $\lambda$, and
$\lambda^\circ$ is the indicator function of the Schubert cell.)  This is
consistent with \cite[Example 8.3]{feher.rimanyi:csmdeg}.
\qede\end{example} If the parabolic subgroup $P$ is the Borel subgroup
$B$, Corollary~\ref{cor:SSMaltgp} is equivalent to the positivity of
CSM classes of Schubert cells,
\cite[Corollary~1.4]{AMSS:shadows}. Indeed, in this case
\[
\ssm(X(u)^\circ, G/B) = (-1)^{\ell(u)} \chc_*(\one_{X(u)^\circ})
\]
as shown in~\cite[Corollary~7.4]{AMSS:shadows}. This equality does not
hold for more general flag manifolds $G/P$, and its proof relies on
additional properties relating the CSM/SSM classes to Demazure-Lusztig
operators from the Hecke algebra
\cite{aluffi.mihalcea:eqcsm,AMSS:shadows}. From this prospective, SSM
classes appear to have a simpler behavior than the CSM classes, and
one can obtain positivity-type statements for a larger class of
varieties.

\subsection{Complements of hyperplane arrangements}\label{ss:hyp}
A typical example of an affine embedding $U \subseteq X$ is the
complement $U=X\smallsetminus D$ of a hypersurface $D\subseteq Z:=X$.
In particular, one can consider a projective hyperplane arrangement
$\cA$ in complex projective space $X=\Pbb^n$, with $A:=D$ the union of
hyperplanes and $U=\Pbb^n\smallsetminus A$ its complement. In this
particular case Theorem~\ref{thm:mather}(b) recovers a consequence of
the following result of \cite[Corollary~3.2]{aluffi:hyperplane}:
\[
\csm(U)=\pi_{\widehat{\cA}}\left(\frac{-h}{1+h}\right)\cap 
\left( c(T\Pbb^n)\cap [\Pbb^n]\right) \:.
\]
Here $h$ denotes the hyperplane class in $\Pbb^n$, $\widehat{\cA}$ is
the corresponding `central arrangement' in~$\C^{n+1}$ with
$\widehat{A}$ its union of linear hyperplanes, and
$\pi_{\widehat{\cA}}$ denotes the corresponding `Poincar\'e
polynomial' of $\widehat{\cA}$ (see
e.g., \cite[p.~1880]{aluffi:hyperplane}):
\[
\pi_{\widehat{\cA}}(t)= \sum_{k=0}^{n+1} \: rk\:
H^k(\C^{n+1}\smallsetminus \widehat{A},\Q)t^k\:.
\]
In particular the Poincar\'e polynomial $\pi_{\widehat{\cA}}$ has
non-negative coefficients and constant term one, and
\[
(-1)^{\dim U}\chssm(U,\Pbb^n)=\pi_{\widehat{\cA}}
\left(\frac{h}{1-h}\right)\cap [\Pbb^n]
\]
is effective.


\section{Donaldson-Thomas type invariants}\label{ss:DT} 
Let $Z\subseteq X$ be a closed reduced subscheme of $X$ as before.
K.~Behrend (\cite[Definition~1.4, Proposition~4.16]{MR2600874})
defines a constructible function $\nu_Z$ and proves that if $Z$ is
proper, then the dimension-$0$ component of $c_*(\nu_Z)$ equals the
corresponding virtual fundamental class $[Z]^{\text{vir}}$, a
`Donaldson-Thomas type invariant' in the terminology
of~\cite[p.~1308]{MR2600874}.

The characteristic cycle of Behrend's constructible
  function $\nu_Z$ is effective because the intrinsic normal cone of
  $Z$ is effective.  More explicitly, $\nu_Z$ is defined as a weighted
  sum
\[
\nu_Z:=\sum_Y (-1)^{\dim Y }\mult(Y) \cdot \Eu_Y\quad,
\]
where the summation is over the supports $Y$ of the components of the
intrinsic normal cone of $Z$ and $\mult(Y)$ is the multiplicity of the
corresponding component. (Cf.~\cite[Definition~1.4]{MR2600874}.) Since
these multiplicities are positive, $\CC(\nu_Z)$ is effective,
cf.~\eqref{eq:eqconst}.

\begin{corol}\label{prop:DT}
Let $X$ be a complex nonsingular variety with globally generated
tangent bundle, and let $Z\subseteq X$ be a closed reduced subscheme.
Then $\chSM(\nu_Z,X)\in A_*(Z)$ is effective. If $TX$ is trivial, then
$\chc_*(\nu_Z)$ is effective.
\end{corol}

In particular for $X$ an abelian variety this implies that
$[Z]^{\text{vir}}$ is non-negative and hence
\[
\chi_{vir}(Z):=\chi(Z,\nu_Z) = deg([Z]^{\text{vir}})\ge 0\:.
\]

\section{Characteristic classes from intersection cohomology}\label{s:charIH}
Another particular case of interest is the characteristic class
of the intersection cohomology sheaf complex. If $Z \subseteq X$ is a
closed subvariety, let $\caIC_Z \in Perv(Z)$ denote the intersection
cohomology complex associated to $Z$. This is the key example of a
perverse sheaf on $Z$,
cf.~\cite{goresky-macpherson},~\cite[Definition~8.2.13]{HTT}
or~\cite[p.~385]{schurmann:book}. The associated constructible
function is $\IC_Z:= \chi_{stalk( \caIC_Z)}$.  We let
\[ 
\cIH(Z):= (-1)^{\dim Z} c_*(\IC_Z) \/,
\] 
be the {\em intersection homology Chern class\/} of $Z$. 
{(Note that $\cIH(Z)$ is an element of the Chow group of $Z$,
not of its intersection homology.)} The sign is
introduced in order to ensure that $\cIH(Z)=c(TZ)\cap [Z]$ if $Z$ is
nonsingular\footnote{More generally, $\cIH(Z)=\csm(Z)$ if $Z$ is a
  rational homology manifold (e.g., $Z$ has only quotient
  singularities).  In fact a quasi-isomorphism $\caIC_Z\simeq
  \C_Z[\dim Z]$ characterizes a rational homology manifold~$Z$
  \cite[p.~34]{borho.macpherson}, see
  also~\cite[Proposition~8.2.21]{HTT}.}.

Similarly, if $Z$ has a {\em small\/} resolution of singularities $f:
Y\to Z$, then $\caIC_Z\simeq Rf_*\C_Y[\dim Y]$ (see
e.g.,~\cite[Example 6.0.9, p.~400]{schurmann:book}) so that
\[ 
\cIH(Z)=f_*(c(TY)\cap [Y]) 
\] 
in this case. This class corresponds to the constructible function $(-1)^{\dim Z}
f_*(\one_Y)$.~We also consider the signed version, $\chcIH(Z)$. For
$Z$ complete, the degree of the zero-dimensional component of
$\cIH(Z)$ equals the `intersection homology Euler characteristic' of
$Z$, $\chiIH(Z)=(-1)^{\dim Z}\chi(Z,\IC_Z)$.  By the functoriality of
$c_*$ and $\chi_{stalk}$, $\chiIH(Z)$ agrees with the intersection
homology Euler characteristic defined as alternating sum of ranks of
intersection homology groups, as in e.g.,~\cite{elduque.geske.maxim}.

More generally, let $f: Y\to Z$ be a
proper morphism. Using that $\chi_{stalk}$ commutes with
$Rf_*$~\cite[\S2.3]{schurmann:book} one may define
$\varphi:=\chi_{stalk}(Rf_*\C_Y)=f_*(\one_Y)$; then
$c_*(\varphi)=f_*c_*(\one_Y)$ by the functoriality of $c_*$.

Theorem~\ref{thm:ssmpos} and Proposition~\ref{prop:holon} imply the
following result.

\begin{corol}\label{prop:abelIC}
Let $Z$ be a closed subvariety of a smooth variety $X$ with $TX$
globally generated. Then $\chSM(\IC_Z,X)$ is effective.  If $TX$ is
trivial (e.g., $X$ is an abelian variety), then 
$(-1)^{\dim Z}\chcIH(Z)=\chcsm(\IC_Z)$ is effective.
\end{corol}

In particular for $X$ an abelian variety this implies that $(-1)^{\dim
  Z}\chiIH(Z) \ge 0$,
recovering~\cite[Theorem~5.3]{elduque.geske.maxim}. In case $Z$ has a
{\em small resolution\/} $f: Y\to Z$,
\[ 
(-1)^{\dim Z}\chcIH(Z)=f_*(c(T^*Y)\cap [Y])=(-1)^{\dim Y}
\chc_*(f_*\one_Y) 
\]
{and this class is effective by~Corollary~\ref{prop:abelIC}.}

\section{Effective charactersitic cycles (II)}\label{s:poschar} 
In this section we collect several instances of positive
characteristic cycles. Some of these were already used in the previous
sections, and they will be {reproduced here for completeness.}

\begin{prop}\label{prop:ICeff} 
Let $X$ be a complex nonsingular variety, and let $Z\subseteq X$ be a
closed reduced subscheme. If $\varphi \in \conF(Z)$ is the
constructible function in one of the cases listed below, then
$\CC(\varphi)$ is an effective cycle.
\begin{enumerate} 
\item[(a)] $\varphi = (-1)^{\dim Z} \Eu_Z$ for $Z$ a closed subvariety
  of $X$.
\item[(b)] $\varphi = \nu_Z$ (Behrend's constructible function,
  see~\S\ref{ss:DT}).
\item[(c)] $\varphi = \chi_{stalk}(\cF)$ for a non-trivial perverse
  sheaf $\cF \in \Perv(Z)$, e.g.:
\begin{enumerate}
\item[(c1)] $\varphi= \IC_Z$ for $Z$ a closed subvariety of $X$, see
  \S\ref{s:charIH}.
\item[(c2)] $\varphi= (-1)^{\dim Z} \one_Z$ for $Z$ pure-dimensional
  and smooth, or more generally a rational homology manifold.
\item[(c3)] $\varphi= (-1)^{\dim Z} \one_Z$ for $Z$ pure-dimensional
  with only local complete intersection singularities (i.e., $Z
  \hookrightarrow X$ is a regular embedding).
\item[(c4)] $\varphi= (-1)^{\dim Y} f_*\one_Y$ for a proper surjective
  semi-small morphism of varieties $f:Y \to Z$, with $Y$ a rational
  homology manifold and $Z$ a closed subvariety of $X$. 
  (See~\S\ref{ss:ssmall} for more on semi-small maps.) 
\item[(c5)] $\varphi = (-1)^{\dim U} \one_U$, where $U \subseteq Z$ is
  a (not necessarily closed) subvariety, such that the inclusion
  $U\hookrightarrow Z$ is an affine morphism and $\C_U[\dim U]$ is a
  perverse sheaf on $U$ (e.g., $U$ is smooth, a rational homology
  manifold, or with only local complete intersection singularities).
\end{enumerate}
\end{enumerate} 
\end{prop}

\begin{proof} 
Part (a) follows from equation \eqref{eq:CC}; part (b) from the
discussion in \S\ref{ss:DT}; part(c) from Proposition~\ref{prop:holon};
part (c1) is discussed in \S\ref{s:charIH}; part (c5)
{combines}
Theorem~\ref{thm:mather} and Remark~\ref{rmk:perverseU}. The
remaining statements are proved as follows:

(c2): $\C_Z[\dim Z]$ is a perverse sheaf for $Z$
  pure-dimensional and smooth, with
\[
\chi_{stalk}(\C_Z[\dim Z])=(-1)^{\dim Z} \one_Z
\]
by definition.  The corresponding regular holonomic $\caD$-module is
just $\cO_Z$. Similarly, $\C_Z[\dim Z]$ is a perverse sheaf for $Z$
pure-dimensional and a rational homology manifold, since then
$\caIC_Z\simeq \C_Z[\dim Z]$; cf.~\cite[p.~34]{borho.macpherson}
or~\cite[Proposition~8.2.21]{HTT}.

(c3): $\C_Z[\dim Z]$ is a perverse sheaf for $Z$
  pure-dimensional with only local complete intersection
  singularities~\cite[Example 6.0.11, p.~404]{schurmann:book}.
  
(c4): In the given hypotheses, the push-forward
$Rf_*\C_Y[\dim Y]$ is a perverse sheaf;
\newline  cf.~\cite[Example 6.0.9, p.~400]{schurmann:book} 
  or~\cite[Definition~8.2.30]{HTT}.  Of course
\[
\chi_{stalk}(Rf_*\C_Y[\dim Y]) = (-1)^{\dim Y} f_*\one_Y \:,
\]
since $f$ is proper.
\end{proof}

To check that for a constructible function $\varphi$, the coefficients
$a_Y$ in the expansion $\CC(\varphi)=\sum_Y a_Y [T^*_Y X]$ are
nonnegative, one can also use the following description
of~$\CC(\varphi)$ in terms of `stratified Morse theory for
constructible functions' from~\cite{schurmann:lectures,
  schurmann.tibar} or~\cite[\S5.0.3]{schurmann:book}:
\[
\CC(\varphi)=\sum_S \; (-1)^{\dim S}\cdot \chi(NMD(S),\varphi)
\cdot [\overline{T^*_SX}]
\]
if $\varphi$ is constructible with respect to a complex algebraic
Whitney stratification of $Z$ with connected smooth strata $S$.  Here
$ \chi(NMD(S),\varphi)$ is the Euler characteristic of a corresponding
{\em normal Morse datum\/} $NMD(S)$ weighted by $\varphi$.  Then
$\CC(\varphi)$ is non-negative (resp., effective) if and only if
\[
(-1)^{\dim S}\cdot \chi(NMD(S),\varphi)\geq 0
\]
for all $S$ (and, resp.,~$(-1)^{\dim S'}\cdot \chi(NMD(S'),\varphi)>0$
for at least one stratum $S'$). If the complex of sheaves $\cF$ is
constructible with respect this complex algebraic Whitney
stratification of $Z$, then one gets for $\varphi:=\chi_{stalk}(\cF)$
and $x\in S$~\cite[(5.38) on p.~294]{schurmann:book}:
\[
\chi(NMD(\cF,x)[-\dim S]) = (-1)^{\dim S}\cdot \chi(NMD(S),\varphi)\:.
\] 
This leads to a direct proof of Proposition~\ref{prop:holon} in the case
of perverse sheaves,
without using $\caD$-modules: $NMD(\cF,x)[-\dim S]$ as above is
concentrated in degree zero for all $S$ if and only if $\mathcal{F}$
is a perverse sheaf ~\cite[Remark 6.0.4, p.~389]{schurmann:book}. This
argument also shows that the condition that $\CC(\chi_{stalk}(\cF))$
be {\em effective\/} is much weaker than the condition that $\cF$ be
perverse.

We end this section by listing some basic operations of constructible
functions which preserve the property of having an effective 
characteristic cycle. These operations may be used to construct 
many more examples to which our Theorem~\ref{thm:ssmpos} 
applies.

\begin{prop}\label{prop:new effective} 
Let $Z$ be a closed reduced subscheme of a nonsingular complex
algebraic variety $X$, and assume that $\varphi$ is a constructible
function on $Z$ with $\CC(\varphi)$ effective.
\begin{enumerate}
\item Let $Z'$ be a closed reduced subscheme of a nonsingular variety
  $X'$, with $\CC(\varphi')$ effective. Then $\CC(\varphi\boxtimes
  \varphi')$ is also effective for the constructible function
  $\varphi\boxtimes \varphi'$ on $Z\times Z'$ defined by
\[
(\varphi\boxtimes \varphi')(z,z'):=\varphi(z)\cdot\varphi'(z')\:.
\]
\item 
Let $f: Z\to Z'$ be a {\em finite\/} morphism, i.e., $f$ is
  proper with finite fibers, with $Z'$ a closed reduced subscheme of a
  nonsingular complex algebraic variety $X'$. Then $f_*(\varphi)$ is a
  constructible function on $Z'$ with $\CC(f_*(\varphi))$
  effective. Here
\[
f_*(\varphi)(z'):=\sum_{z\in f^{-1}(z')}\: \varphi(z)\:.
\]
\item Let $f: X'\to X$ be a morphism of nonsingular complex algebraic
  varieties such that $f: Z':=f^{-1}(Z)\to Z$ is a {\em smooth\/}
  morphism of relative dimension $d$.  Then $(-1)^d
  f^*(\varphi)=(-1)^d \varphi\circ f$ is a constructible function on
  $Z'$ with effective characteristic cycle.
\item Let $f: X\to \C$ be a morphism and let $D$ be the hypersurface
  $\{f=0\}$.  Denote by $\psi_f: \conF(Z)\to \conF(Z\cap D)$ the
  corresponding specialization of constructible
  functions~\cite[\S2.4.7]{schurmann:lectures}. Here
\[
\psi_f(\varphi)(x):=\chi(M_{f|Z,x},\varphi)\:,
\]
with $M_{f|Z,x}$ a {\em local Milnor fiber\/} of $f|Z$ at $x\in Z\cap
D$.  Then $\CC(-\psi_f(\varphi))$ is {\em non-negative.\/} It is {\em
  effective\/} in case $\varphi\neq 0$ has a presentation as
in~\eqref{eq:eqconst} and at least one~$Y$ with $a_Y>0$ is not
contained in $D$.
\item Let $f: X'\to X$ be a morphism of nonsingular complex algebraic
  varieties such that $f$ is {\em non-characteristic\/} with respect
  to the support $\supp(\CC(\varphi))$ of the characteristic cycle of
  $\varphi$ (e.g., $f$ is {\em transversal\/} to all strata $S$ of a
  complex algebraic Whitney stratification of $Z$ for which $\varphi$
  is constructible). Let $d:=\dim X' - \dim X$. Then $(-1)^d
  f^*(\varphi)=(-1)^d \varphi\circ f$ is a constructible function on
  $Z':=f^{-1}(Z)$ with {effective} characteristic cycle.
\end{enumerate}
\end{prop}

\begin{proof} These results can be deduced  from the following facts:

(1) $\CC(\varphi\boxtimes\varphi')=\CC(\varphi)\boxtimes
  \CC(\varphi')$, which follows from $\Eu_Y\boxtimes
  \Eu_{Y'}=\Eu_{Y\times Y'}$ \cite{macpherson:chern}, or from
  stratified Morse theory for constructible functions or
  sheaves~\cite[(5.6) on p.~277]{schurmann:book}:
\[
\chi(NMD(S),\varphi)\cdot  \chi(NMD(S'),\varphi')= \chi(NMD(S\times S'),
\varphi\boxtimes \varphi') \:.
\]

(2) Using the graph embedding, we can assume that the finite map $f:
  Z\to Z'$ is induced from a submersion $f: X\to X'$ of ambient
  nonsingular varieties.  Consider the induced correspondence of
  cotangent bundles:
\[
\begin{CD}
T^*X @< df << f^*T^*X' @> \tau >> T^*X'\:.
\end{CD}
\]
Here $df$ is a closed embedding (since $f$ is a submersion), and 
\[
\tau: df^{-1}(\supp(\CC(\varphi)))\to T^*X'|Z'
\]
is finite, since $f: Z\to Z'$ is finite.  Now $\tau (df^{-1}
\supp(\CC(\varphi)))$ is known to be contained in a conic Lagrangian
subset of $T^*X'|Z'$ (e.g., coming from a stratification of
$f$~\cite[(4.16) on p.~249]{schurmann:book}). Therefore its dimension
is bounded from above by~$\dim X'$. Then also the dimension of
$df^{-1}(\supp(\CC(\varphi)))$ is bounded from above by~$\dim X'$ by
the finiteness of $\tau$, so that
\[ 
df^*(\CC(\varphi))=\CC(\varphi) \cap [f^*T^*X'] 
\]
is a {\em proper intersection.\/} But then
$\tau_*(df^*(\CC(\varphi)))$ is an {\em effective\/} cycle on
$T^*X'|Z'$, and~\cite[\S4.6]{schurmann:lectures}:
\[
\tau_*(df^*(\CC(\varphi)))=\CC(f_*(\varphi))\:.
\]

(3) This follows from $f^*\Eu_Y=\Eu_{f^{-1}(Y)}$ for $Y$ a closed
  subvariety in $Z$.  This can be checked locally, e.g., for $f:
  Z\times Y'\to Z$ the projection along a smooth factor $Y'$, with
  $f^*\Eu_Y=\Eu_Y\boxtimes \one_{Y'} = \Eu_Y\boxtimes \Eu_{Y'}$.

(4) Again it is enough to consider $\chEu_Y:=(-1)^{\dim Y}\Eu_Y$ for
  some subvariety $Y$ of $Z$. If $Y\subseteq \{f=0\}$, then
  $\psi_f(\chEu_Y)=0$ by definition.  So we can assume $Y\not\subseteq
  \{f=0\}$. Then $\CC(-\psi_f(\chEu_Y))$ is 
  by~\cite[Theorem~4.3]{sabbah:quelques} the (Lagrangian) specialization 
  of the
  relative conormal space $[T^*_{f|Z}X]$ along the hypersurface
  $\{f=0\}$, so that it is also effective.

(5) Consider again the induced correspondence of cotangent bundles:
\[
\begin{CD}
T^*X' @< df << f^*T^*X @> \tau >> T^*X\:.
\end{CD}
\]
Then by definition, $f$ is {\em non-characteristic\/} with respect to
the support $\supp(\CC(\varphi))$ of the characteristic cycle of
$\varphi$ if and only if
\[
df: \tau^{-1}(\supp(\CC(\varphi))) \to T^*X'
\]
is {\em proper\/} and therefore finite, 
cf.~\cite[Lemma~3.2]{schurmann:transversality} or~\cite[Lemma~4.3.1,
  p.~255]{schurmann:book}. If $f$ is non-characteristic, then
\[
\CC((-1)^df^*(\varphi)) =df_*( \tau^*(\CC(\varphi))) 
\]
by~\cite[Theorem~3.3]{schurmann:transversality}, and this cycle is
{\em effective\/} if $\CC(\varphi)$ is effective. Indeed the proof 
of~\cite[Theorem~3.3]{schurmann:transversality} is done in two steps:
first for a submersion, where our claim follows from the case~(3)
above; then the case of a closed embedding of a nonsingular subvariety
is (locally) reduced by induction to the case of a hypersurface of
codimension one (locally) given by an equation $\{f=0\}$. Here it is
deduced from case~(4) above, with $Y\not\subseteq \{f=0\}$ by the
`non-characteristic' assumption if $[T^*_YX]$ appears with positive
multiplicity in $\CC(\varphi)$.  
\end{proof}
\section{Further applications}\label{sec:further}

In this final section we explain further applications of
Theorem~\ref{thm:ssmpos} and Proposition~\ref{prop:holon} via the
theory of perverse sheaves.

\subsection{Semi-small maps}\label{ss:ssmall}  
Recall that a morphism $f: Y \to Z$ is called {\em semi-small\/} if
for all~$i>0$,
\[ 
\dim \{z\in Z|\;\dim f^{-1}(z)\geq i\} \leq \dim Z -2i \/;
\] 
the morphism $f$ is {\em small\/} if in addition all inequalities are
strict for $i>0$. See \cite[p.~30]{borho.macpherson},
\cite[Example~6.0.9, p.~400]{schurmann:book}, or
\cite[Definition~8.2.29]{HTT}.

\begin{prop}\label{prop:semi-small} 
Let $f: Y\to Z$ be a proper surjective semi-small morphism of
varieties, with $Y$ a rational homology manifold and $Z$ a closed
subvariety of a smooth variety $X$ with $TX$ globally generated. Then
$(-1)^{\dim Y}\chSM(f_*\one_Y,X)$ is effective.

In particular, if $TX$ is trivial then $(-1)^{\dim
  Y}\chc_*(f_*\one_Y)=(-1)^{\dim Y}f_*\chcsm(Y)$ is effective.  If
moreover $X$ is complete (i.e., an abelian variety) then $(-1)^{\dim
  Y}\chi(Y) \ge 0$.
\end{prop}

The simplest example of a 
{proper\/}
semi-small map $f: Y\to Z:=f(Y)\subseteq X$
is a closed embedding. 
A smooth projective variety $Y$ has a proper
semi-small morphism (onto its image) into an abelian variety $X$ if
and only if its Albanese morphism $alb_X: X\to Alb(X)$ is semi-small
(onto its image)~\cite[Remark 1.3]{LMW:generic}. The corresponding
signed Euler characteristic bound $(-1)^{\dim Y}\chi(Y) \ge 0$ is
further refined in~\cite[Corollary~5.2]{popa.schnell}.  As an example,
if $C$ is a smooth curve of genus $g\geq 3$ and $X$ is its Jacobian,
then the induced Abel-Jacobi map $C^d\to C^{(d)}\to X$ (with $C^{(d)}$
the corresponding symmetric product) is semi-small (onto its image)
for $1\leq d\leq g-1$~\cite[Corollary~12]{weissauer}.

\begin{proof}[Proof of Proposition~\ref{prop:semi-small}]
By the given hypotheses, the push-forward
$Rf_*\C_Y[\dim Y]$ is a perverse sheaf;
  cf.~\cite[Example 6.0.9, p.~400]{schurmann:book} 
  or~\cite[Definition~8.2.30]{HTT}.  Further
\[
\chi_{stalk}(Rf_*\C_Y[\dim Y]) = (-1)^{\dim Y} f_*\one_Y \:,
\]
since $f$ is proper.
{The statement follows then from Proposition~\ref{prop:holon}.}
\end{proof}

\subsection{Regular embeddings and Milnor classes}\label{Milnor classes}
For this application, assume that $Z\subseteq X$ is a regular
embedding, as in \cite[Appendix~B.7]{fulton:IT}. For instance, $Z$
could be a smooth closed subvariety, a hypersurface, or a local
complete intersection in $X$.

\begin{prop}\label{prop:abel-lci}
Let $X$ be a complex nonsingular variety such that $TX$ is globally
generated, and let $Z \subseteq X$ be a regular embedding.  Then
$(-1)^{\dim Z}\chssm(Z,X)$ is effective.  If~$TX$ is trivial, then
$(-1)^{\dim Z}\chcsm(Z)$ is effective.
\end{prop}

\begin{proof} 
The hypothesis imply that $\C_Z[\dim Z]$ is a perverse sheaf, by
Proposition~\ref{prop:ICeff}. Then the claim follows from
Theorem~\ref{thm:ssmpos} and Proposition~\ref{prop:holon}.
\end{proof}

In particular for $X$ an abelian variety this implies that $(-1)^{\dim
  Z}\chi(Z) \ge 0$, recovering~\cite[Theorem~5.4]{elduque.geske.maxim}.

{If $Z\subseteq X$ is a closed embedding, with $X$ smooth, then
the class
\[
\cful(Z):=c(TX)\cap s(Z,X) \in A_*(Z)
\]
is the {\em Chern-Fulton class\/} of $Z$; this is another intrinsic Chern class of
$Z$~\cite[Example~4.2.6]{fulton:IT}. If $Z$ is a local complete intersection in
$X$, then the normal cone $C_ZX=N_ZX$ is a vector bundle, and we have
\[
\cful(Z)=c(TX)\cap\left(c(N_ZX)^{-1}\cap [Z]\right)\in A_*(Z)\,;
\]
in this case, this class is also called the {\em virtual Chern class\/} of $Z$.}
If $Z$ is smooth, $N_ZX$
is the usual normal bundle, so that $\cful(Z)=c(TZ)\cap
[Z]=\csm(Z)$. In general, for singular~$Z$, these classes can be
different, and their difference\footnote{There are different sign
  conventions in the literature. Here we adopt the convention used in
  the original definition of Milnor classes, \cite{PP:hypersurface}.}
\[
\Mil(Z):=(-1)^{\dim Z}\left(\cful(Z)-\csm(Z)\right)\in A_*(Z)
\]
is called the {\em Milnor class\/} of $Z$. Let $\sMil(Z,X)$ be the
corresponding {\em Segre-Milnor class\/}
\begin{align*}
\sMil(Z,X)&:=c(TX)^{-1}\cap \Mil(Z)= (-1)^{\dim Z}\left(c(N_ZX)^{-1}
\cap [Z] - \ssm(Z,X)\right)\\
&\hphantom{:}= (-1)^{\dim Z}\left(s(Z,X) - \ssm(Z,X)\right)  \/.
\end{align*}
As before consider the associated signed classes $\chMil(Z)$ and
\[
\chsMil(Z,X):= c(T^*X|_Z)^{-1}\cap \chMil(Z) \:.
\]

Assume now that $X$ is projective, with very ample line bundle 
$L$, and that
\[
Z=\{s_j=0|\:j=1,\dots,r\}
\]
is the global complete intersection of codimension $r>0$ defined by
sections $s_j\in 
\Gamma(X,L^{\otimes a_j})$ 
for suitable positive
integers $a_j$. Then~\cite[Theorem~1 and
  Corollary~1]{maxim.schurmann.saito} implies that
$\Mil(Z)=c_*(\varphi)$ for a constructible function $\varphi$
associated to a perverse sheaf supported on the singular locus
$Z_{sing}$ of $Z$.\footnote{More precisely,~\cite{maxim.schurmann.saito}
studies the {\em Hirzebruch-Milnor class\/} $M_y(Z)$ of $Z$, 
measuring the difference between the {\em virtual\/} and the 
{\em motivic Hirzebruch class\/} $T^{vir}_{y*}(Z)$ and $T_{y*}(Z)$ of $Z$.
Specializing to $y=-1$ shows (\cite[Corollary~1]{maxim.schurmann.saito})
that $\Mil(Z)=c_*(\varphi)$ for $\varphi$ the constructible
function associated with the underlying perverse sheaf of a mixed Hodge 
module $\cM(s'_1,\dots,s'_r)$ determined by the sections $s_1,\dots,s_r$.}

Applying Proposition~\ref{prop:holon} we obtain:
\begin{corol}\label{cor:Milnor} 
Let $X$ be a smooth projective variety
with $TX$ globally generated, and let
the subvariety $Z=\{s_j=0|\:j=1,\dots,r\}\subseteq X$ be a global
complete intersection as described above.  Then 
$\chsMil(Z,X)\in A_*(Z)$ is non-negative.  If $X$ is
an Abelian variety, then $\chMil(Z)\in
A_*(Z)$ is non-negative.
\end{corol}

As an illustration, if $Z$ has only isolated singularities, then one
can consider the specialization $y=-1$
from~\cite[Corollary~2]{maxim.schurmann.saito} to deduce:
\[
\varphi=\sum_{z\in Z_{sing}}\: \mu_z\cdot \one_z \quad 
\text{so that} \quad \Mil(Z)=c_*(\varphi)= \sum_{z\in Z_{sing}}\: 
\mu_z \cdot [z]\:,
\]
with $\mu_z>0$ the corresponding {\em Milnor number\/} of the isolated
complete interesction singularity $z\in Z_{sing}$.  Here the last
formula for the Milnor class $\Mil(Z)$ is due to~\cite{seade.suwa:ici,
  suwa:ici}.  Therefore in this case, $\sMil(Z,X)=\chsMil(Z)=\Mil(Z) =
\chMil(Z)$.

\subsection{Semi-abelian varieties} \label{Semi-abelian}
Recall that a {\em semi-abelian\/} variety $G$ is a group scheme given
as an extension
\[
0\to \T \to G\to A\to 0
\]
of an abelian variety $A$ by a torus $\T\simeq (\C^*)^n$ ($n\geq 0$),
so that
\[
G\simeq L_1^0\times_A \cdots \times_A L_n^0 \to A
\]
for some degree-zero line bundles $L_i$ over $A$, with $L_i^0$ the
open complement of the zero-section in the total space $L_i\to A$
(cf.~\cite[(5.5)]{MR1769729} or~\cite[p.~12]{liu-maxim-wang}). Then
the projection $p: G\to A$ has the following important stability
property:\smallskip

\begin{quote}
(stab): The group homomorphism $p_*: \conF(X')\to \conF(X)$ induced by
  the morphism $p: X'\to X$ maps the image $\im(\chi_{stalk}:
  \Perv(X')\to \conF(X'))$ to $\im(\chi_{stalk}: \Perv(X)\to
  \conF(X))$.
\end{quote}\smallskip
\noindent Note that the constant morphism $p: X'\to pt$ satisfies the
property~(stab) if and only if $X'$~has the following Euler characteristic
property:
\begin{equation}\label{Euler} 
\chi(X',\cF)\geq 0 \quad \text{for all perverse sheaves 
$\cF\in \Perv(X')$} \:,
\end{equation}
since $\Z_{\ge 0}=\ im(\chi_{stalk}: \Perv(pt)\to \conF(pt)) \subseteq
\conF(pt) =\Z$.  In particular an abelian variety 
satisfies~\eqref{Euler} by Theorem~\ref{thm:ssmpos} and
Proposition~\ref{prop:holon}.

\begin{prop}\label{prop:stab} 
The class of morphisms satisfying property~(stab) is closed under
composition. Further, the following morphisms satisfy property~(stab):
\begin{enumerate}
\item[(a)] $p: X'\to X$ is an affine morphism with all fibers
  zero-dimensional (e.g., a finite morphism or an affine inclusion).
\item[(b)] $p: X'\to X$ is an affine smooth morphism of relative
  dimension one, with all fibers non-empty, connected and of the same
  non-positive Euler characteristic $\chi_p\leq 0$.
\end{enumerate} 
\end{prop}

\begin{example}
The following morphisms $p: X'\to X$ are affine smooth morphisms of
relative dimension one, with all fibers non-empty, connected and 
{constant} non-positive Euler characteristic $\chi_p\leq 0$:
\begin{enumerate}
\item $p: L^0\to X$ is the open complement of the zero-section in the
  total space of a line bundle $L\to X$.
\item $p$ is the projection $p: X\times C\to X$ of a product with a
  smooth non-empty, connected affine curve $C$ with non-positive Euler
  characteristic $\chi(C)\leq 0$.
\item More generally, let $p: X'\to X$ be an {\em elementary
  fibration\/} in the sense of M.~Artin (cf.~\cite[Definition 1.1,
  p.~105]{andre-baldassarri}), i.e., such that it can be factorized as an open
  inclusion $j: X'\to \overline{X'}$ followed by a projective smooth
  morphism of relative dimension one $\bar{p}: \overline{X'}\to X$
  with irreducible (or connected) fibers, such that the induced map of
  the reduced complement $\bar{p}: Z:= \overline{X'}\smallsetminus
  X'\to X$ is a surjective \'{e}tale covering. Then $p: X'\to X$ is an
  affine morphism~\cite[Lemma~1.1.2, p.~106]{andre-baldassarri}. If
  $X$ is connected, then the genus $g\geq 0$ of the fibers of
  $\bar{p}$ and the degree $n\geq 1$ of the covering $\bar{p}: Z\to X$
  are constant, so that all fibers of $p$ have the same Euler
  characteristic $\chi_p=2-2g -n$.  The final assumption
  $\chi_p=2-2g -n\leq 0$ just means $(g,n)\neq (0,1)$, i.e., only the
  affine line $\A^1(\C)$ (with $\chi(\A^1(\C))=1$) is not allowed as a
  fiber of $p$.
\end{enumerate}
\end{example}

The stabilization property (stab) is preserved by compositions,
so that the projection
\[
G\simeq L_1^0\times_A \cdots \times_A L_n^0  \to L_2^0\times_A 
\cdots \times_A L_n^0 \to \cdots \to  L_n^0 \to A
\]
for a semi-abelian variety has the property (stab) by the first
example above. Similarly for the composition of `elementary
fibrations'
\[
\begin{CD}
X'_n @> p_n >> X'_{n-1} @> p_{n-1} >> \cdots @> p_2 >> X'_1 @> p_1 >>
X
\end{CD}
\]
over a connected base $X$, with all fiber Euler characteristics
$\chi_{p_i}\leq 0$.

\begin{corol}\label{corol.semi-abelian}
Assume that the morphism $p: X'\to X$ has the property (stab), with
$X$ a smooth variety such that $TX$ is globally generated.  Let
$\cF\in \Perv(X')$ be a perverse sheaf on~$X'$, with
$\varphi:=\chi_{stalk}(\cF)$.  Then $\chSM(p_*(\varphi))$ is
non-negative.  If~$TX$ is trivial, then $\chc_*(p_*(\varphi)) $ is
non-negative. In particular
$\chi(X',\cF)=\chi(X',\varphi)=\chi(X,p_*(\varphi))\geq 0$ if $X$ is
an  abelian variety.
\end{corol}

The Euler characteristic property~\eqref{Euler} for a semi-abelian
variety $G$ is due to~\cite[Corollary~1.4]{MR1769729} (but their proof
uses results about characteristic cycles on suitable compactifications
and does not extend to the more general context considered above).
The Euler characteristic property~\eqref{Euler} for an algebraic torus
$\T\simeq (\C^*)^n$ is due to~\cite[Corollary~3.4.4]{gabber-loeser} in
the $\ell$-adic context as an application of the {\em generic
  vanishing theorem\/}:
\[
H^i(\T,\cF\otimes L) = 0 \quad \text{for $i\neq 0$}
\]
for a given perverse sheaf $\cF\in \Perv(\T)$ and a {\em generic} rank
one local system $L$ on $\T$.  See 
also~\cite[Theorem~1.2]{liu-maxim-wang:mellin} 
resp.,~\cite[Theorem~1.1]{LMW:generic}
and~\cite[Theorem~1.2]{liu-maxim-wang} for the {\em generic vanishing
  theorem\/} for complex tori, resp., semi-abelian varieties and
algebraically constructible perverse sheaves in the classical topology
(as used in this paper).  The following proof of
Proposition~\ref{prop:stab}(b) is an adaption to the language of
constructible functions of techniques used in these references for
their proofs of the generic vanishing theorem. In this way we can
prove the Euler characteristic property~\eqref{Euler} in a much more
general context, e.g., for any product of connected smooth affine
curves different from the affine line $\A^1(\C)$ (instead of complex
tori).

We close the paper with the proof of Proposition~\ref{prop:stab}.

\begin{proof}[Proof of Proposition~\ref{prop:stab}]
 Note that
$\chi_{stalk}: \Perv(X)\to \conF(X)$ is {\em additive\/} in the sense
that $\chi_{stalk}(\cF)=\chi_{stalk}(\cF')+\chi_{stalk}(\cF'')$ for
any short exact sequence
\[ 
0\to \cF' \to \cF \to \cF''\to 0
\]
in the abelian category $\Perv(X)$. In particular
$\chi_{stalk}(\cF'\oplus \cF'')=\chi_{stalk}(\cF')+\chi_{stalk}(\cF')$
and the zero-sheaf is mapped to the zero-function. Therefore
$\chi_{stalk}$ induces a map from the corresponding Grothendieck group
$\chi_{stalk}: K_0(\Perv(X))\to \conF(X)$, and
\[
\im(\chi_{stalk}: \Perv(X)\to \conF(X))
\]  
is a {\em submonoid\/} of the abelian group $\conF(X)$.  Moreover,
$\chi_{stalk}$ commutes with both pushforwards $Rf_!, Rf_*$ for a
morphism $f: X'\to X$~\cite[\S2.3]{schurmann:book}, with
\[
f_!=f_*: K_0(\Perv(X'))\to K_0(\Perv(X)) \quad \text{and} 
\quad f_!=f_*: \conF(X')\to \conF(X)
\]
in this complex algebraic context~\cite[(6.41), (6.42),
  p.~413]{schurmann:book}. In particular, the pushforward for
constructible functions is functorial with
$\chi(X',\cF)=\chi(X',\varphi)=\chi(X,f_*(\varphi))$ for
$\varphi=\chi_{stalk}(\cF)$ and $\cF\in \Perv(X')$.  Also, this shows
that the property (stab) is preserved by compositions, as claimed in
Proposition~\ref{prop:stab}. The other parts of 
Proposition~\ref{prop:stab} are proved as follows.

(a) An affine morphism $p: X'\to X$ with 
  zero-dimensional fibers induces {\em exact\/}
  functors $Rp_!, Rp_*: \Perv(X')\to \Perv(X)$~\cite[Corollary~6.0.5,
    p.~397 and Theorem~6.0.4, p.~409]{schurmann:book}.

(b) Let $p: X'\to X$ be an affine smooth morphism of relative
  dimension one, with all fibers non-empty, connected and with the same
  non-positive Euler characteristic $\chi_p\leq 0$. Then the shifted
  pullback $p^*[1]: \Perv(X)\to \Perv(X')$ is {\em exact\/}, since $p$
  is smooth of relative dimension one~\cite[Lemma~6.0.3,
    p.~386]{schurmann:book}. Note that $Rp_*$ is not necessarily exact
  for the perverse t-structure. Nevertheless, since $p$ is affine of
  relative dimension one, the perverse cohomology sheaves
  $\;^{m}\cH^i(Rp_*\cF)$ vanish for $i\neq -1,0$ for every perverse
  sheaf $\cF\in \Perv(X')$ \cite[Corollary~6.0.5, p.~397 and
    Theorem~6.0.4, p.~409]{schurmann:book}.  Moreover, the abelian
  category $\Perv(X')$ is a {\em length category,\/} i.e., it is
  noetherian and artinian, so that $\cF\in \Perv(X')$ is a finite
  iterated extension of {\em simple\/} perverse sheaves on
  $X'$~\cite[Theorem~4.3.1, p.~112]{BBD}. By the additivity of
  $\chi_{stalk}$, it is enough to consider a {\em simple\/} perverse
  sheaf $\cF$ on $X'$. If $\;^{m}\cH^{-1}(Rp_*\cF)=0$, then $Rp_*F$ is
  also perverse, with
\[
p_*(\chi_{stalk}(\cF))=\chi_{stalk}(Rp_*\cF) \in \im(\chi_{stalk}: 
\Perv(X)\to \conF(X)) \:.
\] 
Assume now that $\;^{m}\cH^{-1}(Rp_*\cF)\neq 0 \in \Perv(X)$. Then
also 
\[
p^*\left( ^{m}\cH^{-1}(Rp_*\cF) \right) [1] \neq 0 \in \Perv(X')
\]
by the surjectivity of $p$. Since the fibers of $p$ are
non-empty and connected, one gets by~\cite[Corollary~4.2.6.2,
  p.~111]{BBD} a {\em monomorphism\/}
\[
0\to p^*\left( ^{m}\cH^{-1}(Rp_*\cF) \right) [1] \to \cF \:.
\]
This has to be an isomorphism $ p^*\left( ^{m}\cH^{-1}(Rp_*\cF)
\right) [1] \simeq \cF$, since $\cF$ is simple. 
As mentioned before, $Rp_!$ and $Rp_*$ induce the same
constructible function under $\chi_{stalk}$, and the stalk of $Rp_!$
calculates the compactly supported cohomology in the corresponding
fiber.  But
\[ 
\chi_{stalk}(\cF)=\chi_{stalk}\left( p^*\left( ^{m}\cH^{-1}(Rp_*\cF) 
\right) [1] \right) = -p^*(\varphi')
\]
is constant along the fibers of $p$, with 
\[
\varphi':=\chi_{stalk}( ^{m}\cH^{-1}(Rp_*\cF) ) \in  \im(\chi_{stalk}: 
\Perv(X)\to \conF(X))\:.
\]
Finally, all fibers of $p$ have by assumption the same non-positive
Euler characteristic $\chi_p\leq 0$, so that
\[
p_*(\chi_{stalk}(\cF)) =  - p_*(p^*(\varphi')) = -\chi_p \cdot 
\varphi' \in  \im(\chi_{stalk}: \Perv(X)\to \conF(X))\:,
\] 
since $\im(\chi_{stalk}: \Perv(X)\to \conF(X))$ is a submonoid of 
the abelian group $\conF(X)$.  
\end{proof}

\bibliographystyle{halpha}
\bibliography{SSMv2biblio}

\begin{thebibliography}{LMW18b}

\bibitem[AB01]{andre-baldassarri}
Yves Andr\'{e} and Francesco Baldassarri.
\newblock {\em De {R}ham cohomology of differential modules on algebraic
  varieties}, volume 189 of {\em Progress in Mathematics}.
\newblock Birkh\"{a}user Verlag, Basel, 2001.

\bibitem[Alu03]{aluffi:IE1}
Paolo Aluffi.
\newblock Inclusion-exclusion and {S}egre classes. {II}.
\newblock In {\em Topics in algebraic and noncommutative geometry
  ({L}uminy/{A}nnapolis, {MD}, 2001)}, volume 324 of {\em Contemp. Math.},
  pages 51--61. Amer. Math. Soc., Providence, RI, 2003.

\bibitem[Alu13]{aluffi:hyperplane}
Paolo Aluffi.
\newblock Grothendieck classes and {C}hern classes of hyperplane arrangements.
\newblock {\em Int. Math. Res. Not. IMRN}, (8):1873--1900, 2013.

\bibitem[AM09]{aluffi.mihalcea:csm}
Paolo Aluffi and Leonardo~Constantin Mihalcea.
\newblock Chern classes of {S}chubert cells and varieties.
\newblock {\em J. Algebraic Geom.}, 18(1):63--100, 2009.

\bibitem[AM16]{aluffi.mihalcea:eqcsm}
Paolo Aluffi and Leonardo~Constantin Mihalcea.
\newblock Chern--{S}chwartz--{M}ac{P}herson classes for {S}chubert cells in
  flag manifolds.
\newblock {\em Compos. Math.}, 152(12):2603--2625, 2016.

\bibitem[AMSS17]{AMSS:shadows}
Paolo Aluffi, Leonardo~C Mihalcea, J\"{o}rg Sch\"{u}rmann, and Changjian Su.
\newblock Shadows of characteristic cycles, {V}erma modules, and positivity of
  {C}hern-{S}chwartz-{M}ac{P}herson classes of {S}chubert cells.
\newblock {\em arXiv preprint arXiv:1709.08697}, 2017.

\bibitem[BB81]{beilinson.bernstein:localisation}
Alexandre Be{\u\i}linson and Joseph Bernstein.
\newblock Localisation de {$g$}-modules.
\newblock {\em C. R. Acad. Sci. Paris S\'er. I Math.}, 292(1):15--18, 1981.

\bibitem[BB98]{MR1642745}
P.~Bressler and J.-L. Brylinski.
\newblock On the singularities of theta divisors on {J}acobians.
\newblock {\em J. Algebraic Geom.}, 7(4):781--796, 1998.

\bibitem[BBD82]{BBD}
A.~A. Be\u{\i}linson, J.~Bernstein, and P.~Deligne.
\newblock Faisceaux pervers.
\newblock In {\em Analysis and topology on singular spaces, {I} ({L}uminy,
  1981)}, volume 100 of {\em Ast\'{e}risque}, pages 5--171. Soc. Math. France,
  Paris, 1982.

\bibitem[BDK81]{MR647684}
Jean-Luc Brylinski, Alberto~S. Dubson, and Masaki Kashiwara.
\newblock Formule de l'indice pour modules holonomes et obstruction d'{E}uler
  locale.
\newblock {\em C. R. Acad. Sci. Paris S\'er. I Math.}, 293(12):573--576, 1981.

\bibitem[Beh09]{MR2600874}
Kai Behrend.
\newblock Donaldson-{T}homas type invariants via microlocal geometry.
\newblock {\em Ann. of Math. (2)}, 170(3):1307--1338, 2009.

\bibitem[BF97]{MR1451256}
Brian~D. Boe and Joseph H.~G. Fu.
\newblock Characteristic cycles in {H}ermitian symmetric spaces.
\newblock {\em Canad. J. Math.}, 49(3):417--467, 1997.

\bibitem[BFL90]{MR1084458}
P.~Bressler, M.~Finkelberg, and V.~Lunts.
\newblock Vanishing cycles on {G}rassmannians.
\newblock {\em Duke Math. J.}, 61(3):763--777, 1990.

\bibitem[BK81]{brylinski.kashiwara:KL}
J.-L. Brylinski and M.~Kashiwara.
\newblock Kazhdan-{L}usztig conjecture and holonomic systems.
\newblock {\em Invent. Math.}, 64(3):387--410, 1981.

\bibitem[BLS00]{MR1783853}
J.-P. Brasselet, D{\~u}ng~Tr{\'a}ng L{\^e}, and J.~Seade.
\newblock Euler obstruction and indices of vector fields.
\newblock {\em Topology}, 39(6):1193--1208, 2000.

\bibitem[BM83]{borho.macpherson}
Walter Borho and Robert MacPherson.
\newblock Partial resolutions of nilpotent varieties.
\newblock In {\em Analysis and topology on singular spaces, {II}, {III}
  ({L}uminy, 1981)}, volume 101 of {\em Ast\'{e}risque}, pages 23--74. Soc.
  Math. France, Paris, 1983.

\bibitem[BR62]{borel.remmert:uber}
A.~Borel and R.~Remmert.
\newblock \"{U}ber kompakte homogene {K}\"{a}hlersche {M}annigfaltigkeiten.
\newblock {\em Math. Ann.}, 145:429--439, 1961/1962.

\bibitem[Bri05]{brion:flagv}
Michel Brion.
\newblock Lectures on the geometry of flag varieties.
\newblock In {\em Topics in cohomological studies of algebraic varieties},
  Trends Math., pages 33--85. Birkh\"auser, Basel, 2005.

\bibitem[Bri12]{brion:spherical}
Michel Brion.
\newblock Spherical varieties.
\newblock In {\em Highlights in {L}ie algebraic methods}, volume 295 of {\em
  Progr. Math.}, pages 3--24. Birkh\"{a}user/Springer, New York, 2012.

\bibitem[EGM18]{elduque.geske.maxim}
Eva Elduque, Christian Geske, and Laurentiu Maxim.
\newblock On the signed {E}uler characteristic property for subvarieties of
  abelian varieties.
\newblock {\em J. Singul.}, 17:368--387, 2018.

\bibitem[Ern94]{MR1301184}
Lars Ernstr{\"o}m.
\newblock Topological {R}adon transforms and the local {E}uler obstruction.
\newblock {\em Duke Math. J.}, 76(1):1--21, 1994.

\bibitem[FK00]{MR1769729}
J.~Franecki and M.~Kapranov.
\newblock The {G}auss map and a noncompact {R}iemann-{R}och formula for
  constructible sheaves on semiabelian varieties.
\newblock {\em Duke Math. J.}, 104(1):171--180, 2000.

\bibitem[FR18]{feher.rimanyi:csmdeg}
L\'{a}szl\'{o} Feh\'{e}r and Rich\'{a}rd Rim\'{a}nyi.
\newblock Chern-{S}chwartz-{M}ac{P}herson classes of degeneracy loci.
\newblock {\em Geom. Topol.}, 22(6):3575--3622, 2018.

\bibitem[FRW18]{feher2018motivic}
L{a}szlo~M Feher, Rich{a}rd Rimanyi, and Andrzej Weber.
\newblock Motivic chern classes and {K}-theoretic stable envelopes.
\newblock {\em arXiv preprint arXiv:1802.01503}, 2018.

\bibitem[Ful84]{fulton:IT}
William Fulton.
\newblock {\em Intersection theory}.
\newblock Springer-Verlag, Berlin, 1984.

\bibitem[Gin86]{ginzburg:characteristic}
Victor Ginzburg.
\newblock Characteristic varieties and vanishing cycles.
\newblock {\em Invent. Math.}, 84(2):327--402, 1986.

\bibitem[GL96]{gabber-loeser}
Ofer Gabber and Fran\c{c}ois Loeser.
\newblock Faisceaux pervers {$l$}-adiques sur un tore.
\newblock {\em Duke Math. J.}, 83(3):501--606, 1996.

\bibitem[GM83]{goresky-macpherson}
Mark Goresky and Robert MacPherson.
\newblock Intersection homology. {II}.
\newblock {\em Invent. Math.}, 72(1):77--129, 1983.

\bibitem[HTT08]{HTT}
Ryoshi Hotta, Kiyoshi Takeuchi, and Toshiyuki Tanisaki.
\newblock {\em {$D$}-modules, perverse sheaves, and representation theory},
  volume 236 of {\em Progress in Mathematics}.
\newblock Birkh\"auser Boston, Inc., Boston, MA, 2008.

\bibitem[KT84]{kashiwara.tanisaki:characteristic}
M.~Kashiwara and T.~Tanisaki.
\newblock The characteristic cycles of holonomic systems on a flag manifold
  related to the {W}eyl group algebra.
\newblock {\em Invent. Math.}, 77(1):185--198, 1984.

\bibitem[LMW18a]{liu-maxim-wang:mellin}
Yongqiang Liu, Laurentiu Maxim, and Botong Wang.
\newblock Mellin transformation, propagation, and abelian duality spaces.
\newblock {\em Adv. Math.}, 335:231--260, 2018.

\bibitem[LMW18b]{liu-maxim-wang}
Yongqiang Liu, Laurentiu Maxim, and Botong Wang.
\newblock Perverse sheaves on semi-abelian varieties.
\newblock {\em arXiv preprint arXiv:1804.05014}, 2018.

\bibitem[LMW19]{LMW:generic}
Yongqiang Liu, Laurentiu Maxim, and Botong Wang.
\newblock Generic vanishing for semi-abelian varieties and integral {A}lexander
  modules.
\newblock {\em Math. Z.}, 293(1-2):629--645, 2019.

\bibitem[LT81]{MR634426}
D{\~u}ng~Tr{\'a}ng L{\^e} and Bernard Teissier.
\newblock Vari\'et\'es polaires locales et classes de {C}hern des vari\'et\'es
  singuli\`eres.
\newblock {\em Ann. of Math. (2)}, 114(3):457--491, 1981.

\bibitem[Mac74]{macpherson:chern}
R.~D. MacPherson.
\newblock Chern classes for singular algebraic varieties.
\newblock {\em Ann. of Math. (2)}, 100:423--432, 1974.

\bibitem[MSS13]{maxim.schurmann.saito}
Laurentiu Maxim, Morihiko Saito, and J\"{o}rg Sch\"{u}rmann.
\newblock Hirzebruch-{M}ilnor classes of complete intersections.
\newblock {\em Adv. Math.}, 241:220--245, 2013.

\bibitem[Ohm06]{ohmoto:eqcsm}
Toru Ohmoto.
\newblock Equivariant {C}hern classes of singular algebraic varieties with
  group actions.
\newblock {\em Math. Proc. Cambridge Philos. Soc.}, 140(1):115--134, 2006.

\bibitem[PP01]{PP:hypersurface}
Adam Parusi{\'n}ski and Piotr Pragacz.
\newblock Characteristic classes of hypersurfaces and characteristic cycles.
\newblock {\em J. Algebraic Geom.}, 10(1):63--79, 2001.

\bibitem[PS13]{popa.schnell}
Mihnea Popa and Christian Schnell.
\newblock Generic vanishing theory via mixed {H}odge modules.
\newblock {\em Forum Math. Sigma}, 1:e1, 60, 2013.

\bibitem[Sab85]{sabbah:quelques}
Claude Sabbah.
\newblock Quelques remarques sur la g\'eom\'etrie des espaces conormaux.
\newblock {\em Ast\'erisque}, (130):161--192, 1985.

\bibitem[Sch03]{schurmann:book}
J{\"o}rg Sch{\"u}rmann.
\newblock {\em Topology of singular spaces and constructible sheaves},
  volume~63 of {\em Instytut Matematyczny Polskiej Akademii Nauk. Monografie
  Matematyczne (New Series) [Mathematics Institute of the Polish Academy of
  Sciences. Mathematical Monographs (New Series)]}.
\newblock Birkh{\"a}user Verlag, Basel, 2003.

\bibitem[Sch05]{schurmann:lectures}
J\"org Sch\"urmann.
\newblock Lectures on characteristic classes of constructible functions.
\newblock In {\em Topics in cohomological studies of algebraic varieties},
  Trends Math., pages 175--201. Birkh\"auser, Basel, 2005.
\newblock Notes by Piotr Pragacz and Andrzej Weber.

\bibitem[Sch17]{schurmann:transversality}
J{\"o}rg Sch{\"u}rmann.
\newblock Chern classes and transversality for singular spaces.
\newblock In {\em Singularities in Geometry, Topology, Foliations and
  Dynamics}, Trends in Mathematics, pages 207--231. Birkh{\"a}user, Basel,
  2017.

\bibitem[SS98]{seade.suwa:ici}
Jos\'{e} Seade and Tatsuo Suwa.
\newblock An adjunction formula for local complete intersections.
\newblock {\em Internat. J. Math.}, 9(6):759--768, 1998.

\bibitem[ST10]{schurmann.tibar}
J\"{o}rg Sch\"{u}rmann and Mihai Tib\u{a}r.
\newblock Index formula for {M}ac{P}herson cycles of affine algebraic
  varieties.
\newblock {\em Tohoku Math. J. (2)}, 62(1):29--44, 2010.

\bibitem[Suw97]{suwa:ici}
Tatsuo Suwa.
\newblock Classes de {C}hern des intersections compl\`etes locales.
\newblock {\em C. R. Acad. Sci. Paris S\'{e}r. I Math.}, 324(1):67--70, 1997.

\bibitem[Wei06]{weissauer}
Rainer Weissauer.
\newblock Brill-{N}oether sheaves.
\newblock {\em arXiv preprint arXiv:math/0610923}, 2006.

\bibitem[Zha18]{zhang:chern}
Xiping Zhang.
\newblock Chern classes and characteristic cycles of determinantal varieties.
\newblock {\em J. Algebra}, 497:55--91, 2018.

\end{thebibliography}

\end{document}